\documentclass[10 pt,oneside,reqno,a4paper]{amsart} 
\usepackage{amsfonts,amssymb,amscd,amsmath, enumerate, verbatim, calc} 
\usepackage{float}
\usepackage{amsthm}
\newtheorem{theorem}{Theorem}[section]
\newtheorem{definition}{Definition}[section]
\usepackage[all]{xy}
\usepackage{tikz-cd}
\newtheorem{lemma}[theorem]{Lemma}
\newtheorem{corollary}[theorem]{Corollary}
\newtheorem{proposition}[theorem]{Proposition}
\newtheorem{remark}[theorem]{Remark}
\usepackage{mathtools}

\newtheorem{example}[theorem]{Example}
\numberwithin{equation}{section}

\usepackage{amsfonts}
\newcommand{\field}[1]{\mathbb{#1}}          

\newcommand{\N}{\field{N}}
                   
\newcommand{\C}{\field{C}}

\renewcommand{\ker}{\textnormal{ker}}

\begin{document}
	\title{Representation and Extending structure of a Multiplicative Lie algebra}
    \author{ Neeraj Kumar Maurya$^{1}$ and Sumit Kumar Upadhyay$^{2}$ \vspace{.4cm}\\
				{$^{1,2}$Department of Applied Sciences,\\ Indian Institute of Information Technology Allahabad\\Prayagraj, U. P., India} }
	
		\thanks{$^1$neerajkr2699@gmail.com, $^2$upadhyaysumit365@gmail.com }

\begin{abstract}
In this paper, we attempt to develop a general notion of representation theory for multiplicative Lie algebras. Consequently, we study equivalent, irreducible, and completely reducible representations. We also introduce the notion of the action of a multiplicative Lie algebra on a vector space and show that there is a one-to-one correspondence between the action and the representation. We also studied the extended structure of a multiplicative Lie algebra through a group.
\end{abstract}
	\maketitle

\section{Introduction}
In $1993$, G. J. Ellis \cite{GJ} introduced the concept of multiplicative Lie algebras. Since multiplicative Lie algebras generalize both groups and Lie algebras, many structural properties like nilpotency, solvability, Lie nilpotency, Lie solvability, homology theory, non-abelian tensor product, Schur multiplier, and Bogomolov multiplier, analogous to those of groups and Lie algebras, have been studied in \cite{AGNM, GNM, GNMA, GM, KJ, RS, MRS} by many authors. Representation theory provides a classical method for investigating the structure of abstract algebraic objects in terms of linear groups. A fully developed representation theory of groups and Lie algebras is available (for details, one may see the books \cite{EH, BS}) and has played a central role in the development of modern algebra. In $2025$, A. Kumar et al. \cite{KU} attempted to give a notion of linear representations of a multiplicative Lie algebra, but they noticed that it coincides with group and Lie algebra representations. They have also studied the possible multiplicative Lie algebra structures on the general linear group. Since there is no general concept of representation of a multiplicative Lie algebra available so far, it is natural to seek a general notion of representation of a multiplicative Lie algebra. It is also known that representations of a group and Lie algebras on vector spaces are equivalent to their linear actions on vector spaces, respectively. G. Donadze et al. \cite{GNM, GNMA} introduced the concept of an action of a multiplicative Lie algebra over a multiplicative Lie algebra.  Motivated by the correspondence of action and representation, and the notion of action developed in \cite{GNM, GNMA}, the main aim of this paper is to introduce a foundational framework for the representation of multiplicative Lie algebras and establish general results concerning the structure and properties of such representations.

Besides representation theory, extension and factorization problems provide another important direction in the study of multiplicative Lie algebras. The homology theory for multiplicative Lie algebras has been studied in \cite{AGNM}. 
One of the important problems in multiplicative Lie algebras is the classification of all multiplicative Lie algebras $G$ (up to isomorphism) having $H$ as an ideal such that $G/H$ is isomorphic to $K$, for any two multiplicative Lie algebras $H$ and $K$. Motivated by this problem, in 2019, R. Lal et al. \cite{RS} developed the Schreier extension theory and defined the central extensions for multiplicative Lie algebras. In $2021,$ M. S. Pandey et al. \cite{MS} developed the Schreier extension theory for the center extension and Lie center extension of multiplicative Lie algebras. With the help of work done in  \cite{MS}, D. Pal et al. \cite{DSS} gave a method to determine the multiplicative Lie algebra structures on the semi-direct product of groups under certain conditions.

In $2014$, A. Agore et al. \cite{AG1, AG} developed the extending structures problem for groups and Lie algebras, respectively. This motivates us to introduce the concept of extending structures problem ``\textit{Determine all multiplicative Lie algebras $G$ (up to isomorphism) having $H$ as a normal subalgebra such that $G/H$ is isomorphic to $K$ (as a group)",} for a multiplicative Lie algebra $H$ and a group $K$, which we refer to as the extending structures problem of a multiplicative Lie algebra through a group. The concept of extending structure problems may be helpful for classifying multiplicative Lie algebra structures on a group. Some work has been done in \cite{KU, KP, MKU, PK, MS1, GW} towards classifying multiplicative Lie algebra structures.

The paper is organized as follows. In Section 2, we recall the necessary definitions and preliminary results concerning multiplicative Lie algebras. In Section 3, we introduce the notion of a representation of a multiplicative Lie algebra and establish its fundamental properties. In the subsequent sections, we investigate structural aspects of these representations, including invariant subspaces and decomposability, and develop further consequences of the proposed definition. In Section 4, we discuss the extending structures problem of a multiplicative Lie algebra through a group.

\section{Preliminaries}
In this section, we recall some basic definitions that are useful to this article.   
\begin{definition}\cite{GJ}
	A multiplicative Lie algebra is a triple $ (G,\cdot,\star), $ where $ (G,\cdot) $ is a group together with a binary operation $ \star $ on $G$ satisfying the following identities: 
	\begin{enumerate}
		\item $ x\star x=1 $,  
		\item $ x\star(y\cdot z)=(x\star y)\cdot{^y(x\star z)} $, 
		\item $ (x\cdot y)\star z= {^x(y\star z)}\cdot (x\star z) $, 
		\item $ ((x\star y)\star {^yz})\cdot((y\star z)\star{^zx})\cdot((z\star x)\star{^xy})=1 $, 
		\item $ ^z(x\star y)=(^zx\star {^zy})$, 
	\end{enumerate}
for all $x,y,z\in G$, where $^xy$ denotes $x\cdot y\cdot x^{-1}$. We say $ \star $ is a multiplicative Lie algebra structure on the group $G$.
\end{definition}
\begin{definition}\cite{GJ, RS}
Let $(G,\cdot,\star)$ be a multiplicative Lie algebra. Then
\begin{enumerate}
\item A subgroup $H$ of $G$ is said to be a subalgebra of G if $x\star y \in H$ for all $x,y \in H$.

\item A subalgebra $H$ of $G$ is said to be an ideal of $G$ if it is a normal subgroup of $G$ and $x\star y \in H$ for all $x \in G$ and $y \in H$. For example, the commutator subgroup $[G, G]$ of $G$ is an ideal of $G$, it is denoted by $G'$. 




\item Let $(G',\circ , \star')$ be another multiplicative Lie algebra. A group homomorphism $\psi: G \to G'$ is called a multiplicative Lie algebra homomorphism if $\psi(x\star y) =\psi(x) \star ' \psi(y)$ for all $x, y \in G$. We say that $\psi$ is an isomorphism if it is a bijective homomorphism. 

\end{enumerate}
\end{definition}

 \begin{example}\cite{RS, MS1}
        \begin{enumerate} 
 \item[1.] On any group $G$, we have a multiplicative Lie algebra structure $\star $ given by  $x \star y = 1$ for all $x, y \in G$, called \textbf{trivial} multiplicative Lie algebra structure.
		 
		\item[2.] On any group $G$, we have a multiplicative Lie algebra structure $\star$ given by  $x \star y = [x,y]$ for all $x, y \in G$, called \textbf{improper} multiplicative Lie algebra structure. It is denoted by $G_{[,]}.$
		
		\item[3.] Let $G$ be a cyclic group. Then there is only trivial multiplicative Lie algebra structure on $G$.
         \item[4.]  There are two distinct  multiplicative Lie algebra structure $\star$ on  Klein four group $V_4 =  \langle \ a, b  \ | \ a^2 = b^2 = 1,\ ab  = ba \ \rangle$ given by:\\
		(i) $a \star b = 1,$ \\
		(ii)  $ a \star a = b \star b = 1, \  a \star b = a.$
		\item[5.] Any Lie ring or Lie algebra is a multiplicative Lie algebra.
            \end{enumerate}
    \end{example} 

\begin{definition} \cite{GNM}
    Let $G$ and $H$ be two multiplicative Lie algebras. By an action of $G$ on $H$ we mean an underlying group action of $G$ on $H$ and $H$ on $G$, together with a map \[ \langle-,-\rangle:G\times H\longrightarrow H, \qquad (x,y)\longmapsto \langle x,y\rangle, \] satisfying the following conditions: \begin{align*} 
    &\langle x,yy'\rangle = \langle x,y\rangle \langle{}^{y}x,~ ^{y}y'\rangle, \\
    &\langle xx',y\rangle = \langle{}^{x}x',{}^{x}y\rangle \langle x,y\rangle, \\
    & \langle x\star x',{}^{x'}y\rangle\, \langle{}^{y}x,\langle x',y\rangle\rangle^{-1}\, \langle ^xx',\langle x,y\rangle^{-1}\rangle^{-1} =1, \\
   & \langle{}^{y'}x,y\star y'\rangle\, ({}^{y}y' \star \langle x,y\rangle\,) ({}^{x}y \star \langle x,y'\rangle^{-1}) =1,
    \end{align*}
    where $x,x'\in G$, $y,y'\in H$, and  ${}^{x}x'=xx'x^{-1}, {}^{y}y'=yy'y^{-1},$ for $x\in G$ and $y\in H$, the symbols ${}^{x}y$ and ${}^{y}x$ denote the given group actions of $G$ on $H$ and $H$ on $G$, respectively.
\end{definition}


  \section{Representation theory of a Multiplicative Lie algebra }
  In this section, we attempt to introduce a notion of a representation of a multiplicative Lie algebra. 

\begin{definition}\label{A} Let $G$ be a multiplicative lie algebra, and let $V$ be a vector space over a field $\mathbb{K}$. A representation of $G$ on $V$ is a pair $(\phi,\psi)$, where  
$\phi : G\to GL(V)$ is a group homomorphism, and $\psi : G \to End(V)$ is a map satisfying the following conditions:
\begin{enumerate}
    \item $\psi_{gg'}(v)= \psi_{^g{g'}}(\phi_g(v)) + \psi_g(v)$,
    \item  $\psi_{{^gg'}}( {\phi_g(v)}) =~  \phi_g({\psi_{g'}({v})})$,
    \item $\psi_{g \star g'}(\phi_{g'}(v))=\psi_{g}(\psi_{g'}(v))- \psi_{^g{g'}}(\psi_g(v))$,
\end{enumerate}
for all $g, g'\in G$ and $v\in V.$

Here, note that $\phi$ is a group representation of the underlying group $G$. We call $(\phi,\psi)$ is an $n$-dimensional representation if $dim(V)= n.$
\end{definition}

Let $(\phi,\psi)$ be an $n$-dimensional representation of $G$ on a vector space $V$ over $\mathbb{K}$. Fixing a basis of $V$, we identify $V$ with $\mathbb{K}^n$ and, consequently, $End(V)$ with the matrix algebra $M_n(\mathbb{K})$. Thus, the matrix representation of $G$ is a pair $(\phi',\psi')$, where 
$\phi' : G\to GL_n(\mathbb{K})$ is a group homomorphism, and $\psi' : G \to M_n(\mathbb{K})$ is a map satisfying the following conditions:
\begin{enumerate}
    \item $\psi'(gg')= \psi'(^g{g'})\phi'(g) + \psi'(g)$,
    \item  $\psi'({^gg'}) \phi'(g) =~  \phi'(g)\psi'(g')$,
    \item $\psi'(g \star g')\phi'(g')=\psi'(g)\psi'(g') - \psi'(^g{g'})\psi'(g)$,
\end{enumerate}
for all $g, g'\in G.$

\begin{proposition}\label{1}
    Let $(\phi,\psi)$ be a representation of a multiplicative Lie algebra $G$. Then
    
    \begin{enumerate}
        \item $\psi_1 = 0_V$, here $0_V$ denotes the zero map.
        \item $\psi_{g^{-1}}(v)= - \psi_{g}(\phi_{{g^{-1}}}(v))$ for all $g \in G$ and $v\in V$.
        \item  For any $g \in G$, $\psi_{g^i}=\psi_{g}(\phi_{g^{i-1}} + \phi_{g^{i-2}}+ \cdots + \phi_g+ I_V)$ for all $i \in \N$, where $I_V$ denotes the identity map on $V$.
    \end{enumerate} 
\end{proposition}

\begin{proof} 
\begin{enumerate}

\item Since $\phi_1 = I_V$, we have
\begin{align*}
   \psi_1(v) &=\psi_{1.1}(v)  = \psi_1(\phi_1(v)) + \psi_1(v) = \psi_1(v)+ \psi_1(v)\\
   \Rightarrow \psi_1(v) &= 0, \forall~  v\in V.
\end{align*}

\item  Since $\psi_1(v)=0$ for all $v\in V$, we have
\begin{align*}
 &0=\psi_{g^{-1}g}(v) =\psi_{g}(\phi_{g^{-1}}(v))+ \psi_{g^{-1}}(v) \\
 \implies & \psi_{g^{-1}}(v)= -\psi_{g}(\phi_{g^{-1}}(v)) ~\text{for all}~  v\in V ~\text{and}~g\in G.
\end{align*}
\item By Definition \ref{A}, we have
\begin{align*}
    \psi_{g^2}(v)=\psi_{gg}(v)&=\psi_{ggg^{-1}}(\phi_g(v)) + \psi_g(v)\\& = \psi_g(\phi_g(v)+ v)\\&= \psi_g(\phi_g + I_V)(v)
\end{align*}  for all $g\in G$ and $v\in V$. Assume that $\psi_{g^k}=\psi_{g}(\phi_{g^{k-1}} + \phi_{g^{k-2}}+ \cdots + \phi_g+ I_V)$. Then 
    \begin{align*}
    \psi_{g^{k+1}}(v)=\psi_{gg^k}(v)&=\psi_{gg^{k}g^{-1}}(\phi_g(v)) + \psi_g(v) \\&= \psi_{g^k}(\phi_g(v)) + \psi_g(v) \\&= \psi_{g}(\phi_{g^{k-1}} + \phi_{g^{k-2}}+ \cdots + \phi_g+ I_V)(\phi_g(v)) + \psi_g(v)\\&=\psi_g(\phi_{g^{k}} + \phi_{g^{k-1}}+ \cdots + \phi_g+ I_V)(v)
\end{align*} for all $g\in G$ and $v\in V$.
\end{enumerate}
\end{proof}

\begin{corollary}\label{3.2}
Let $G$ be a finite multiplicative Lie algebra and $(\phi,\psi)$ be a representation of $G$ such that $\phi_g = I_V$ for every $g\in G$. Suppose the characteristic of the field $\mathbb{K}$ is zero. Then $\psi_g = 0$ for every $g\in G$.
\end{corollary}
\begin{proof}
Let $g\in G$. Since $G$ is finite, $g^m = 1 $ for some $m\in \mathbb{N}$. Thus $\psi_{g^m}(v) = \psi_{1}(v) = 0$. By identity $(3)$ of Proposition \ref{1}, $\psi_{g}(\phi_{g^{m-1}} + \phi_{g^{m-2}}+ \cdots + \phi_g+ I_V) (v) = 0$. Thus $\psi_{g}(mv) = 0$. Therefore, $\psi_g = 0$ for every $g\in G$.
\end{proof}

\begin{remark}\label{2}
\begin{enumerate}
\item If $\phi_g = I_V$ for all $g\in G$, then $\psi : G \to End(V)$ is a multiplicative Lie algebra homomorphism such that $\psi_{^gg'}(v)=\psi_{g'}(v)$ for all $v\in V$, where $End(V)$ is a multiplicative Lie algebra with the multiplicative Lie algebra structure $\star$ defined as 
$(f\star g)(v) = f(g(v))- g(f(v))$ for all $f, g \in End(V)$.
\item If the underlying group structure of $G$ is abelian and $\phi_g = I_V$ for all $g\in G$, then $\psi : G \to End(V)$ is a multiplicative Lie algebra homomorphism. In particular, if $G$ is a Lie algebra over a field $\mathbb{K}$, then the pair $(I_V,\psi)$ is a Lie algebra representation of $G$. 

    \item An one-dimensional representation of a multiplicative Lie algebra $G$ is a pair $(\phi,\psi)$,  where  $\phi : G\to \mathbb{K}^{*}$ is a group homomorphism and $\psi : G \to \mathbb{K}$ is a map satisfying the following conditions:
\begin{enumerate}
    \item $\psi(^gg')= \psi(g')$
    \item $\psi(gg')= \psi({g'})\phi(g) + \psi(g)$
    \item $\psi(g \star g')=0$
\end{enumerate}
for all $g, g'\in G$.
In particular, if the underlying group structure of $G$ is abelian and $\phi(g) = 1$ for all $g\in G$, then   $\psi: G \to \mathbb{K}$ is a multiplicative Lie algebra homomorphism such that $\psi(g \star g') = 0$ for all $g, g'\in G$. 

\item If the underlying group structure of $G$ is abelian with trivial multiplicative Lie algebra structure, then an one-dimensional representation of $G$ is a pair $(\phi,\psi)$,  where  $\phi: G\to \mathbb{K}^{*}$ is a group homomorphism and $\psi: G \to \mathbb{K}$ is map such that $\psi(gg')= \psi(g')\phi(g) + \psi(g)$ for all $g,g'\in G$.
\end{enumerate}

\end{remark}

\begin{example}
  By Proposition \ref{1} and Remark \ref{2}, all one dimensional multiplicative Lie algebra representation of a cyclic group $G$ generated by $g$ is of the form $(\phi, \psi)$, where $\phi: G \to \C^*$ is a group homomorphism and $\psi$ is a map from $G$ to $\C$ such that $\psi(g^k)=\psi(g)(\phi(g)^{k-1}+ \phi(g)^{k-2}+ \cdots + 1),$ for all $k$, $1\leq k\leq n$.
\end{example}

    
\begin{example}
    Consider $G=V_4=<a,b| ~ a^2=e=b^2, ab=ba>$ Klein's four-group with multiplicative Lie algebra structure given by $a\star b=a$.
    Let $(\phi,\psi)$ be an one-dimensional representation of $G$, where $\phi: G\to \C^*$ is defined by $\phi(a)=1,\phi(b)=-1$. Then by Remark \ref{2}, we have  $\psi(e)=0$,  $\psi(a)=0$ and
    $\psi(ab)=\psi(b)\phi(a)+ \psi(a)=\psi(b)$. This shows that, if we take a map $\psi: G \to \C$ such that $\psi(a)=0$, $\psi(b) =\psi(ab)= \lambda $  for any $\lambda \in \C$, then $(\phi,\psi)$ is an one-dimensional representation of $G$.
\end{example}
\begin{example}
    Consider the group $S_3 = \langle a,b ~|~ a^2 = b^3 =1, aba = b^{-1}\rangle$ with improper multiplicative Lie algebra structure. Let $(\phi,\psi)$ be an one-dimensional representation of $G$, where 
$\phi: S_3 \rightarrow \C^*$ defined by $\phi(a)=-1$ and $\phi(b)=1$. Then

 \noindent\textbf{Claim:} $\psi(b)=\psi(b^2)=0$, and $\psi(a) =\psi(ab)=\psi(ab^2)$.
\begin{proof}
  By Proposition \ref{1} and Remark \ref{2}, we have 
$\psi(1)=\psi(b^3)=\psi(b)(\phi(b^2)+ \phi(b)+ 1)=3\psi(b)$ which gives $\psi(b)=0,$ and 
$\psi(b^2)=\psi(b)\phi(b)+ \psi(b)$ implies $\psi(b^2)=0$ and $\psi(ab)=\psi({^ab})\phi(a)+ \psi(a)$ which gives $\psi(ab)=\psi(a)$, similarly, we have $\psi(ab^2)=\psi(a).$  
\end{proof}
  
\end{example}

\begin{example} 
    Consider $G=D_4=<a,b|~ a^2=1=b^4, aba=b^{-1}>$ with multiplicative Lie algebra structure given by $a\star b=b$ $($\cite{MS1}$)$. Let $(\phi,\psi)$ be an one-dimensional representation of $G$, where $\phi:D_4 \to \C^*$ is defined by $\phi(a)=1$ and $\phi(b)=-1$. Then
    
    \noindent\textbf{Claim:} $\psi(x)=0$ for all $x\in D_4$.
    \begin{proof}
   By Proposition \ref{1} and Remark \ref{2}, we have $\psi(a^2)=\psi(a)\phi(a)+ \psi(a)=2 \psi(a)$ which implies $\psi(a)=0$.

  Also, $\psi(ab^2ab^2)=\psi(ab^2)\phi(ab^2)+ \psi(ab^2)=2 \psi(ab^2)$ implies $\psi(ab^2)=0$. 
  
  Now, $\psi(a\star b)\phi(b)=\psi(a)\psi(b)-\psi({^ab})\psi(a)=0$ which gives $\psi(b)=0$. Then by Remark \ref{2}, we have $\psi(b^2)=0=\psi(b^3)$. Now, $\psi(ab)=\psi({^ab})\phi(a)+\psi(a)=0$ implies $\psi(ab)=0$ and $\psi(ab^3)=\psi({^ab^3})\phi(a)+\psi(a)=0$ implies $\psi(ab^3)=0$.
        
    \end{proof}
 
\end{example}

\begin{example}
    Consider $G=D_4=<a,b|~ a^2=1=b^4, aba=b^{-1}>$ with multiplicative Lie algebra structure given by $a\star b=b^2$ $($\cite{MS1}$)$. Let $(\phi,\psi)$ be an one-dimensional representation of $G$.
    Suppose $\phi:D_4 \to \C^*$ is a group homomorphism defined on generators by $\phi(a)=1$ and $\phi(b)=-1$. Then 

    \textbf{Claim:} $\psi(a)=\psi(b^2)= \psi(ab^2)=0$ and $\psi(ab)=\psi(ab^3)=\psi(b)=\psi(b^3) $.
    \begin{proof}
        By Proposition \ref{1} and Definition \ref{A}, we have $\psi(a^2)=\psi(a)\phi(a)+ \psi(a)=2 \psi(a)$ which gives $\psi(a)=0$.
        Similarly, $\psi(ab^2ab^2)=\psi(ab^2)\phi(ab^2)+ \psi(ab^2)=2 \psi(ab^2)$ gives $\psi(ab^2)=0$ and $\psi(a\star b)\phi(b)=\psi(a)\psi(b)-\psi({^ab})\psi(a)=0$ which implies $\psi(b^2)=0$. 
        
        Now, $\psi(ab)=\psi({^ab})\phi(a)+\psi(a)$ implies $\psi(ab)=\psi(b^3)$ and $\psi(ab^3)=\psi({^ab^3})\phi(a)+\psi(a)$ implies $\psi(ab^3)=\psi(b)$ and $\psi(b^3)=\psi(b)(\phi(b^2)+ \phi(b)+ 1)$ implies $\psi(b^3)=\psi(b)$.
    \end{proof}
\end{example}

\begin{example}\label{1.9}
    Consider $G=\mathbb{Z}/4\mathbb{Z}$ with trivial multiplicative Lie algebra structure. Let $(\phi,\psi)$ be a two-dimensional representation of $G$. Suppose $\phi:G \to GL_2(\C)$ is a group homomorphism defined on generator by $\phi(1)=\begin{bmatrix}
i & 0 \\
0 & -i
\end{bmatrix}$. Then 

\textbf{Claim:} $\psi(1)=\begin{bmatrix}
a & 0 \\
0 & d
\end{bmatrix},$ where $a,d\in \C$.
\begin{proof}
    By  identity $(2)$ of Definition \ref{A}, we have 
    $\psi(g')\phi(g)=\phi(g)\psi(g')$ for all $g,g' \in G.$ Now, for $g=g'=1,$ we have
    \begin{align*}
        \psi(1)\phi(1) &=\phi(1)\psi(1)\\
        \begin{bmatrix}
a & b \\
c & d
\end{bmatrix}\begin{bmatrix}
i & 0 \\
0 & -i
\end{bmatrix} &= \begin{bmatrix}
i & 0 \\
0 & -i
\end{bmatrix} \begin{bmatrix}
a & b \\
c & d
\end{bmatrix}\\
    \end{align*} after comparing, we get $c=b=0$. Thus, $\psi(1)=\begin{bmatrix}
a & 0 \\
0 & d
\end{bmatrix},$ where $a,d\in \C$.
\end{proof}

\end{example}

\begin{theorem}
Let $G$ be a multiplicative Lie algebra structure. Then there is a one-to-one correspondence between one-dimensional representations $(\phi,\psi)$ of $G$ and one-dimensional representations $(\widetilde{\phi}, \widetilde{\psi})$ of $G/G'$.
\end{theorem}

\begin{proof}
Let $(\phi,\psi)$ be a one-dimensional representation of $G$. We first show that $G'\subseteq\ker\phi$ and $\psi (g) = 0$, for all $g\in G'$.

Since $\mathbb{C}^{*}$ is abelian and $\phi$ is a group homomorphism, $\phi([g,h])=[\phi(g),\phi(h)]=1$, for all $g,h\in G$, .
Hence $G'\subseteq\ker\phi$.

On the other hand,
\begin{align*}
\psi([g,h])
=\psi(g{}~ ^{h}g^{-1})
=\psi({}^{h}g^{-1})\phi(g)+\psi(g)\
=\psi(g^{-1})\phi(g)+\psi(g)\
=0 ~ \text{by Proposition \ref{1} (2)}.
\end{align*}
Thus, $\psi$ vanishes on every commutator. Moreover, if $x = [a, b] ,y = [c, d]\in G' $, then $\psi (x) = \psi (y) = 0$. Also
\[
\psi(xy)
=\psi({}^{x}y)\phi(x)+\psi(x)
=\psi(y)\phi(x)+\psi(x)=0.
\]
 Therefore, $\psi (g) = 0$, for all $g\in G'$. 

Since $G'\subseteq\ker\phi$, there exists a unique group homomorphism
$\widetilde{\phi}:G/G'\to \mathbb{C}^{*}$
such that $\widetilde{\phi}(gG')=\phi(g)$.

Furthermore, define $\widetilde{\psi}:G/G'\to \mathbb{C}$
by $\widetilde{\psi}(gG')=\psi(g)$.
This map is well defined. Indeed, if $gG'=hG'$, then $g=ah$ for some $a\in G'$. Therefore $\psi(g)=\psi(ah)=\psi({}^{a}h)\phi(a)+\psi(a)=\psi(h),$ since $\phi(a)=1$ and $\psi(a)=0$ for every $a\in G'$.

We now verify that $(\widetilde{\phi},\widetilde{\psi})$ is a one-dimensional representation of $G/G'$. By the identities of Definition \ref{A}, for $g,h\in G$, we have
\begin{align*}
\widetilde{\psi}((gG')(hG'))
=\widetilde{\psi}(ghG')\
=\psi(gh)\
=\psi({}^{g}h)\phi(g)+\psi(g)\
=\widetilde{\psi}({}^{g}hG')\widetilde{\phi}(gG')
+\widetilde{\psi}(gG').
\end{align*}
\begin{align*}
    \widetilde{\psi}({}^{g}hG')
=\psi({}^{g}h)
=\psi(h)
=\widetilde{\psi}(hG').
\end{align*}
\begin{align*}
    \widetilde{\psi}((gG')\widetilde{\star}(hG'))
=\widetilde{\psi}((g\star h)G')
=\psi(g\star h)=0.
\end{align*}

 Hence $(\widetilde{\phi},\widetilde{\psi})$ is a one-dimensional representation of $G/G'$.

Conversely, let $(\widetilde{\phi},\widetilde{\psi})$ be a one-dimensional representation of $G/G'$. Let $\pi: G \to G/G'$ be the quotient multiplicative Lie algebra homomorphism. Now, take
$\phi=\widetilde{\phi}\circ\pi,
~\psi=\widetilde{\psi}\circ\pi.$ Then, clearly $\phi$ is a group homomorphism.

For $g,h\in G$,
\begin{align*}
\psi(gh)
=\widetilde{\psi}(ghG')\
=\widetilde{\psi}({}^{g}hG')\widetilde{\phi}(gG')
+\widetilde{\psi}(gG')\
=\psi({}^{g}h)\phi(g)+\psi(g).
\end{align*}
Furthermore,
\[
\psi({}^{g}h)
=\widetilde{\psi}({}^{g}hG')
=\widetilde{\psi}(hG')
=\psi(h),
\]
and
\begin{align*}
\psi(g\star h)
=\widetilde{\psi}((g\star h)G')\
=\widetilde{\psi}((gG')\widetilde{\star}(hG'))\
=0.
\end{align*}
Thus $(\phi,\psi)$ is a one-dimensional representation of $G$. This completes the proof.
\end{proof}

\begin{definition}
     Two representation $(\phi^1,\psi^1)$ and $(\phi^2, \psi^2)$ $($where, $\phi^{i} : G \rightarrow GL(V_{i})$ and $\psi^{i}: G \rightarrow End(V_{i}))$ are said to be $\textbf{equivalent}$ if  there exists an isomorphism $T : V_{1} \rightarrow V_{2}$ such that $\phi^2_{g}\circ T=T \circ \phi^1_g$ and $\psi^2_{g}\circ T=T \circ \psi^1_g$ for all $g\in G.$ In that case, we write $(\phi^1,\psi^1) \sim (\phi^2, \psi^2)$.
\end{definition}

\begin{definition}
 Let $(\phi,\psi)$ be a representation of a multiplicative Lie algebra $G$. A subspace $W\leq V$ is $G$-invariant if, $\phi_g(w)\in W$ and $\psi_g(w)\in W$ for all $g\in G$ and $w\in W$.
\end{definition}

For $(\phi,\psi)$ from Example \ref{1.9}, $\C e_1$ and $\C e_2$ both are $\mathbb{Z}/4\mathbb{Z}$-invariant.

\begin{definition}
    Let $(\phi^1,\psi^1)$ and $(\phi^2,\psi^2)$ be two representations of a multiplicative Lie algebra $G$ of degree $n_1$ and $n_2$, respectively. Then their (external) direct sum is a pair $(\phi^1 \oplus \phi^2,\psi^1\oplus \psi^2 )$, where $\phi^1 \oplus \phi^2: G \to GL(V_1 \oplus V_2)$ is given by $$(\phi^1 \oplus \phi^2)_g(v_1, v_2)= (\phi^1_g(v_1),\phi^2_g(v_2))$$ and $\psi^1 \oplus \psi^2: G \to End(V_1 \oplus V_2)$ is given by $$(\psi^1 \oplus \psi^2)_g(v_1, v_2)= (\psi^1_g(v_1),\psi^2_g(v_2)).$$
\end{definition}
\begin{example}
    Let $(\phi^1,\psi^1)$ and $(\phi^2,\psi^2)$ be two representations of $\mathbb{Z}/4\mathbb{Z}$, where $\phi^1: \mathbb{Z}/4\mathbb{Z} \to \C^*$ given by $\phi^1(1)=i$, $\psi^1: \mathbb{Z}/4\mathbb{Z} \to \C$ given by $\psi^1(1)=a$, and $\phi^2: \mathbb{Z}/4\mathbb{Z} \to \C^*$ by $\phi^2(1)=-i$, $\psi^2: \mathbb{Z}/4\mathbb{Z} \to \C$ given by $\psi^2(1)=b$. Then $(\phi^1 \oplus \phi^2)(1)=\begin{bmatrix}
        i & 0 \\ 0 & -i
    \end{bmatrix}$ and $(\psi^1 \oplus \psi^2)(1)=\begin{bmatrix}
        a & 0 \\ 0 & b
    \end{bmatrix}$.
\end{example}

\begin{example}\label{1.12}
    Consider $G=S_3$ with trivial multiplicative Lie algebra. Let $(\phi,\psi)$ be a degree two representation of $S_3$, where $\phi: S_3 \rightarrow GL_2(\C)$ be specified on generators $(12)$ and $(123)$ by \[\phi(12)=
\begin{bmatrix}
-1 & -1 \\
0 & 1
\end{bmatrix}
 , \phi(123)=
\begin{bmatrix}
-1 & -1 \\
1 & 0
\end{bmatrix}.
\] Then 

\textbf{Claim:} $\psi:S_3 \to M_2(\C)$ is zero map, that is, $\psi(g)=\begin{bmatrix}
0 & 0 \\
0 & 0
\end{bmatrix}$ for all $g\in S_3$.

\begin{proof}
Suppose $\psi(12)=\begin{bmatrix}
    a & b  \\ c & d
\end{bmatrix},$ where $a,b,c,d\in \mathbb{C}.$ Since $(12)^{2} = 1$, we have $\psi(12)(\phi(12)+ I_2)=0$ which gives $a=2b, c=2d$. Also,  by identity $(2)$ of Definition \ref{A}, we have $\psi(12)\phi(12)=\phi(12)\psi(12)$ which gives $d=0,$. Thus, $\psi(12)=\begin{bmatrix}
    2b & b \\ 0 & 0
\end{bmatrix}.$

Now, by the property of $\psi$, it is easy to see that 

$$\psi(23)= \phi(13)\psi(12)\phi(12) ~~\text{and}~~ \psi(13)= \phi(23)\psi(12)\phi(23),$$

$$\psi(123)= \psi(23) \phi(13) + \psi(13)       ,$$
which implies $\psi(23)=\begin{bmatrix}
    b & -b \\ -b & b
\end{bmatrix}$, $\psi(13)=\begin{bmatrix}
    0 & 0 \\ b & 2b
\end{bmatrix}$  and $\psi(123)=\begin{bmatrix}
    2b & b \\ -b & b
\end{bmatrix}.$

Now, by identity $(3)$ of Definition \ref{A}, we have 
$\psi(23)\psi(13)=\psi(13)\psi(12)$ which gives $b=0.$ Therefore, $\psi: S_3 \to M_2(\mathbb{C})$ is zero map.
\end{proof}

Let $(\phi',\psi')$ be another degree one representation of $S_3$, where $\phi'(g)=1$ for all $g\in S_3$. Then by corollary \ref{3.2}, we have $\psi':S_3 \to M_2(\C)$ is zero map, that is, $\psi'(g)=0$ for all $g\in S_3$.

Now, their direct product is $(\phi \oplus \phi')(12)=\begin{bmatrix}
        -1 & -1 & 0 \\ 0 & 1 & 0 \\ 0 & 0 &1
    \end{bmatrix}$, $(\phi \oplus \phi')(123)=\begin{bmatrix}
        -1 & -1 & 0 \\ 1 & 0 & 0 \\ 0 & 0 &1
    \end{bmatrix}$ and $(\psi \oplus \psi')(g)=\begin{bmatrix}
        0 & 0 & 0 \\ 0 & 0& 0 \\ 0 & 0 & 0
    \end{bmatrix}$ for all $g\in S_3$.

\end{example}

\begin{definition}
    A representation $(\phi,\psi)$ of a multiplicative Lie algebra $G$ is said to be irreducible if the only $G$-invariant subspaces are $\{0\}$ and $V$.
\end{definition}
\begin{example}
    Any degree one representation $(\phi,\psi)$ of $G$ is irreducible, since $\C$ has no proper non-zero subspaces.
\end{example}
\begin{example}
    The representations from Example \ref{1.9} is not irreducible as $\C e_1$ and $\C e_2$ are $\mathbb{Z}/4\mathbb{Z}$-invariant subspaces for $(\phi,\psi)$.
\end{example}

\begin{proposition}
    If $(\phi,\psi)$ is a representation of degree $2$, then $(\phi,\psi)$ is irreducible if and only if there is no common eigenvector $v$ to all $\phi_g$ and $\psi_g$ with $g\in G$.
\end{proposition}
    \begin{proof}
        Since dim ${V}=2$, any non-zero proper $G$-invariant subspace $W$ is one-dimensional. So $W=\mathbb{\mathbb{K}}v$ for some $v\in V$. Let $x\in G$. Then by $G$-invariance of $W$ we have $\phi_x(v),\psi_x(v) \in W=\mathbb{K}v$. It follows that $v$ must be an eigenvector for all $\phi_x$ and $\psi_x$ with $x\in G.$
    \end{proof}
\begin{example}
    
Consider $G=S_3$ with trivial multiplicative Lie algebra.
    The representation $(\phi,\psi)$ of $S_3$, where $\phi: S_3 \rightarrow GL_2(\C)$ be specified on generators $(12)$ and $(123)$ by \[\phi(12)=
\begin{bmatrix}
-1 & -1 \\
0 & 1
\end{bmatrix}
 , \phi(123)=
\begin{bmatrix}
-1 & -1 \\
1 & 0
\end{bmatrix}
\] and $\psi:S_3 \to M_2(\C)$ is zero map, is irreducible.
\end{example}
\begin{proof}
By the above Proposition, it is sufficient to show that $\phi(12)$ and $\phi(123)$ do not have a common eigenvector. Indeed, direct computation shows that $\phi(12)$ has eigenvalues $1$ and $-1$ with eigen spaces $\mathbb{C}e_1$ and $\mathbb{C}\begin{bmatrix}
        -1 \\ 2
    \end{bmatrix}$, respectively. Clearly, $e_1$ is not an eigenvector of $\phi(123)$ as
    $\phi(123)\begin{bmatrix}
        1 \\ 0
    \end{bmatrix}=\begin{bmatrix}
        -1 \\ 1
    \end{bmatrix}$. Also, $\phi(123)\begin{bmatrix}
        -1 \\ 2
    \end{bmatrix}=\begin{bmatrix}
        -1 \\ -1
    \end{bmatrix}$. Thus, $\phi(123)$ and $\phi(12)$ have no common eigenvector, which implies that $(\phi,\psi)$ is irreducible by the above discussion.
    \end{proof}
\begin{definition}
    Let $G$ be a multiplicative Lie algebra. A representation $(\phi, \psi)$ of $G$ on $V$ is said to be completely reducible if $V = V_1 \oplus V_2 \oplus \cdots \oplus V_n$, where the $V_i$ are $G$-invariant subspaces and $(\phi,\psi)|_{V_i}$ is irreducible for all $i=1,\cdots ,n$.
\end{definition}
 Equivalently, $(\phi,\psi)$ is completely reducible if $(\phi,\psi) \sim (\phi^{(1)}\oplus \phi^{(2)}\oplus \cdots \oplus \phi^{(n)},\psi^{(1)}\oplus \psi^{(2)}\oplus \cdots \oplus \psi^{(n)})$ where the $(\phi^{(i)},\psi^{(i)})$ are irreducible representations.

 \begin{definition}
     A non-trivial representation of a multiplicative Lie algebra $G$ on $V$ is said to be \textit{decomposable} if $V=V_1 \oplus V_2$ with $V_1, V_2$  non-zero $G$-invariant subspaces. Otherwise, $V$ is called \textit{indecomposable}.
 \end{definition}
\begin{lemma}
    Let $(\phi,\psi)$ be a representation of $G$ on $V$. If $(\phi,\psi)$ is equivalent to a decomposable representation. Then $(\phi,\psi)$ is decomposable.
\end{lemma}
\begin{proof}
    Let $(\phi',\psi')$ be a decomposable representation of $G$ on $W$ such that $(\phi,\psi)\sim (\phi',\psi')$. Then there exists a vector space homomorphism $T:V \to W$ such that $$\phi_g' \circ T=T \circ \phi_g ~\text{and}~ \psi_g' \circ T=T \circ \psi_g$$ for all $g\in G.$ Since $(\phi',\psi')$ is decomposable, we have $W=W_1 \oplus W_2$, where $W_1,W_2$ are non-zero $G$-invariant subspaces.
    
    Define $V_1=T^{-1}(W_1)$ and $ V_2=T^{-1}(W_2)$. We claim that $V=V_1 \oplus V_2$. Indeed, if $v\in V_1 \cap V_2$, then $Tv\in W_1 \cap W_2=\{0\}$, but $T$ is injective, so $v=0$. Thus $V_1 \cap V_2=\{0\}$. Now, let $v\in V$, then $Tv=w_1 + w_2$ for some $w_1\in W_1$ and $w_2 \in W_2$. Then $v=T^{-1}(w_1)+ T^{-1}(w_2) \in V_1 + V_2$. So, $V=V_1 \oplus V_2.$
    
    Now, we prove that $V_1$ and $V_2$ are $G$-invariant. Let $v\in V_i,$ then $Tv\in W_i$. Since $W_i$ is $G$-invariant, we have $\phi_g' \circ Tv\in W_i$ and $\psi_g' \circ Tv\in W_i$ implies $T\circ \phi_gv \in W_i$ and $T\circ \psi_gv \in W_i$. Hence, we conclude that each $V_i$ is $G$-invariant.
\end{proof} 
Similarly, we can prove the following Lemma.
\begin{lemma}
    Let $(\phi,\psi)$ be a representation of $G$ on $V$. If $(\phi,\psi)$ is equivalent to an irreducible representation. Then $(\phi,\psi)$ is irreducible.
\end{lemma}
\begin{lemma}
    Let $(\phi,\psi)$ be a representation of $G$ on $V$. If $(\phi,\psi)$ is equivalent to a completely reducible representation. Then $(\phi,\psi)$ is completely reducible.
\end{lemma}

\begin{definition}
    Let $G$ be a multiplicative Lie algebra and $V$ be a finite-dimensional vector space over a field $\mathbb{K}$. We say that $G$ acts on $V$ if there exists a pair $ (\theta, \delta)$ where 
$\theta, \delta: G\times V \to V$ are maps satisfying the following conditions:
\begin{enumerate}
\item $\theta(1, v) = v$,
\item  $\theta(g, v_1 +v_2) = \theta(g, v_1 )+ \theta(g, v_2)$,
\item  $\theta(gg', v) = \theta(g, \theta(g', v) )$,
\item $\delta(g,v_1+v_2)=\delta({g}, v_1) + \delta(g,v_2)$,
\item  $\delta({^gg'}, \theta(g, v)) =~  ^g{\delta({g'}, {v})}$,
    \item $\delta(gg',v)=\delta({^gg'}, \theta(g, v)) + \delta(g,v)$,
    \item $\delta(g \star g', \theta(g', v))= \delta(g,\delta(g',v)) - \delta({^gg'},\delta(g,v))$,
\end{enumerate} 
where  $^gg'=gg'g^{-1}$, for all $g,g' \in G$ and $v, v_1, v_2\in V$. Here, we call $ (\theta, \delta)$ is an action  of $G$ on $V$.
\end{definition}

\begin{theorem}
    Let $G$ be a multiplicative Lie algebra and $V$ be a finite-dimensional vector space over a field $\mathbb{K}$. Then  $G$ acts on $V$ if and only if there exists a unique representation of $G$ on $V$. 
\end{theorem}
\begin{proof}
  Let $ (\theta, \delta)$ be an action  of $G$ on $V$. Then, the pair $(\phi,\psi)$ is a representation of $G$ on $V$, where $\phi: G \to GL(V)$ and $\psi: G \to End(V)$ are defined as $\phi_g(v)= \theta(g,v)$ and $\psi_g(v)= \delta(g,v)$.

     Conversely, suppose that $(\phi,\psi)$ is a representation of $G$ on $V$, that is, we have a group homomorphism $\phi : G\to GL(V)$, and a map $\psi : G \to End(V)$ satisfying the following conditions:
\begin{enumerate}
\item $\psi_{^g{g'}}(\phi_g(v)) = \phi_g(\psi_{{g'}}(v))$,
    \item $\psi_{gg'}(v)= \psi_{g_{g'}}(\phi_g(v)) + \psi_g(v)$,
    \item $\psi_{g \star g'}(\phi_{g'}(v))=\psi_{g}(\psi_{g'}(v))- \psi_{g_{g'}}(\psi_g(v))$.
\end{enumerate} Then we have an action $(\theta, \delta)$ defined as $\theta(g,v)= \phi_g(v)$ and $\delta(g,v)= \psi_g(v)$.
\end{proof}



\section{Extending structures of a multiplicative Lie algebra}
We start the section by introducing a normal subalgebra and a pseudo action. 
\begin{definition}
    A subgroup $H$ of $G$ is said to be a normal subalgebra of $ G$ if it is a normal subgroup and subalgebra of $G$.
\end{definition}

\begin{definition}
    Let $G$ be a multiplicative Lie algebra and $H$ be a group. A pseudo action of $H$ on $G$ is a pair $(\sigma,\delta)$ where $\sigma: G \to Aut(H)$ is a group homomorphism and $\delta: H \times G \to G$ is a map satisfying the following conditions:
    \begin{enumerate}
        \item $\delta(g,xy)=\delta(g,x) ~{^x{\delta({g}, {y})}},$
        \item $\delta(gh,x)= \delta({h}, {x}) \delta(g,x),$
        \item ${^x\delta(g,y)}= \delta(\sigma_x(g),{^xy}),$
        \item $\delta(g \star h, {x})=\delta({^gh},\delta(g,x)^{-1})\delta(\sigma_x(g),\delta(h,x)),$
        
    \end{enumerate} 
    where ${^gh}= ghg^{-1}, {^xy}= xyx ^{-1}$ for all $g,h\in H$ and $x,y \in G.$
\end{definition}

Let $G$ be a group with multiplicative Lie algebra structure $\star$. Let $H$ be a normal subalgebra of $G$ and $S$ be a subgroup of $G$ such that $G = HS$ and $H\cap S = \{1\}.$ In fact, $G = H \rtimes_{\sigma} S$ as a group, where $\sigma: S \to Aut(H)$ is  group homomorphism defined as $\sigma(x)(h) = xhx^{-1}$.  For any $x,y \in S$ and $h\in H$, we have $x \star y, h \star x \in G$. Since $G = HS$, there exist four maps $f: S \times S \to H,~ \widetilde{\star}: S \times S \to S, ~\theta:H \times S \to H$ and $\delta:H \times S \to S$ such that
\[x \star y=f(x,y) (x\widetilde{\star} y)\] and \[ h \star x=\theta(h,x)\delta(h,x).\] 
Now, we discuss the properties of maps $f, \widetilde{\star}, \theta$, and $\delta$. For each $x,y,z \in S$ and $h,k,l \in H$, the following identities hold.
\begin{enumerate}
     \item Since $x\star x=1$, we have 
    \begin{align}
        f(x,x)=1 ~~\text{and}~~ x\widetilde{\star} x=1. 
    \end{align}
    
\item Since $ x\star yz =(x\star y) {^y(x\star z)}$, we have
\begin{align}
    f(x,yz) (x\widetilde{\star} yz) &= f(x,y) (x\widetilde{\star} y) {^y(f(x,z) (x\widetilde{\star} z))} \notag
    \\&= f(x,y) (x\widetilde{\star} y) \sigma_y(f(x,z)){^y (x\widetilde{\star} z)} \notag
    \\&=f(x,y) \sigma_{(x\widetilde{\star} y)y}(f(x,z))(x\widetilde{\star} y){^y(x\widetilde{\star} z)} \notag
\end{align} 
 which implies 
 \begin{align}
      f(x,yz)= f(x,y) \sigma_{(x\widetilde{\star} y)y}(f(x,z)) ~~\text{and}~~ x\widetilde{\star} yz= (x\widetilde{\star} y)~{^y(x\widetilde{\star} z)}.
 \end{align}

\item Since $xy \star z={^x(y \star z)} (x \star z)$, we have 
\begin{align*}
    f(xy,z) (xy\widetilde{\star} z)&= {^x(f(y,z)(y\widetilde{\star} z))} f(x,z)(x\widetilde{\star} z)\notag
    \\&=\sigma_x(f(y,z)){^x(y\widetilde{\star} z)} f(x,z)(x\widetilde{\star} z) \notag
    \\&=\sigma_x(f(y,z)) \sigma_{^x(y\widetilde{\star} z)}(f(x,z)){^x(y\widetilde{\star} z)}(x\widetilde{\star} z)
\end{align*}
    which implies 
    \begin{align}
        f(xy,z)=\sigma_x(f(y,z)) \sigma_{^x(y\widetilde{\star} z)}(f(x,z)) ~~\text{and}~~ (xy\widetilde{\star} z)={^x(y\widetilde{\star} z)}(x\widetilde{\star} z).
    \end{align}
    
\item For the Jacobi identity, we have
\begin{align}
    (x\star y) \star {^yz}&= (f(x,y)(x\widetilde{\star} y)) \star {^yz} \notag
   \\&= {^{f(x,y)}((x\widetilde{\star} y)\star {^yz})}(f(x,y)\star {^yz}) \notag 
  \\&= {^{f(x,y)}(f(x\widetilde{\star} y,{^yz}) ((x\widetilde{\star} y) \widetilde{\star} {^yz})}) \theta(f(x,y), {^yz}) \delta(f(x,y), {^yz}) \notag 
   \\&={^{f(x,y)}(f(x\widetilde{\star} y,{^yz}))} {^{f(x,y)}((x\widetilde{\star} y) \widetilde{\star} {^yz})}) \theta(f(x,y), {^yz}) \delta(f(x,y), {^yz}) \notag 
   \\&= f(x,y) f(x\widetilde{\star} y,{^yz})  \sigma_{((x\widetilde{\star} y) \widetilde{\star} {^yz})}(f(x,y)^{-1}\theta(f(x,y), {^yz})) ((x\widetilde{\star} y) \widetilde {\star} {^yz})\delta(f(x,y), {^yz}) \notag
\end{align}
similarly,
\begin{align}
    (y\star z) \star {^zx}&= f(y,z) f(y\widetilde{\star} z,{^zx})  \sigma_{((y\widetilde{\star} z) \widetilde{\star} {^zx})}(f(y,z)^{-1}\theta(f(y,z), {^zx})) ((y\widetilde{\star} z) \widetilde {\star} {^zx})\delta(f(y,z), {^zx}) \notag \\
    \text{and}~ (z\star x) \star {^xy}&= f(z,x) f(z\widetilde{\star} x,{^xy})  \sigma_{((z\widetilde{\star} x) \widetilde{\star} {^xy})}(f(z,x)^{-1}\theta(f(z,x), {^xy})) ((z\widetilde{\star} x) \widetilde {\star} {^xy})\delta(f(z,x), {^xy}) \notag
\end{align}
since, $((x\star y) \star {^yz})((y\star z) \star {^zx})((z\star x) \star {^xy})=1$, so we have 
\begin{align}
    &f(x,y) f(x\widetilde{\star} y,{^yz})  \sigma_{((x\widetilde{\star} y) \widetilde{\star} {^yz})}(f(x,y)^{-1}\theta(f(x,y), {^yz}))  \notag\\& \sigma_{((x\widetilde{\star} y) \widetilde {\star} {^yz})\delta(f(x,y), {^yz})}(f(y,z) f(y\widetilde{\star} z,{^zx})  \sigma_{((y\widetilde{\star} z) \widetilde{\star} {^zx})}(f(y,z)^{-1}\theta(f(y,z), {^zx}))) \notag \\& \sigma_{((x\widetilde{\star} y) \widetilde {\star} {^yz})\delta(f(x,y), {^yz})  ((y\widetilde{\star} z) \widetilde {\star} {^zx})\delta(f(y,z), {^zx}) }(f(z,x) f(z\widetilde{\star} x,{^xy})  \sigma_{((z\widetilde{\star} x) \widetilde{\star} {^xy})}(f(z,x)^{-1}\theta(f(z,x), {^xy}))) =1 \notag\\
    ~\text{and}~ &((x\widetilde{\star} y) \widetilde {\star} {^yz})\delta(f(x,y), {^yz})  ((y\widetilde{\star} z) \widetilde {\star} {^zx})\delta(f(y,z), {^zx})((z\widetilde{\star} x) \widetilde {\star} {^xy})\delta(f(z,x), {^xy})=1.
\end{align}
 \item Since ${^z(x\star y)}= {^zx} \star {^zy}$, we have 
 \begin{align*}
     {^z(f(x,y)(x\widetilde{\star} y))}&= f({^zx}, {^zy})({^zx} \widetilde{\star} {^zy}) \\
     {\sigma_z(f(x,y))}{^z(x\widetilde{\star} y)}&= f({^zx}, {^zy})({^zx} \widetilde{\star} {^zy})
 \end{align*} 
which implies 
\begin{align}
    \sigma_z(f(x,y))=f({^zx}, {^zy}) ~~\text{and}~~ {^z(x\widetilde{\star} y)}=({^zx} \widetilde{\star} {^zy}).
\end{align}





\item Since $h \star xy= (h \star x) ^x(h \star y)$, we have 
\begin{align*}
    \theta(h,xy) \delta(h,xy) &= (\theta(h,x)\delta(h,x)) ~{^x(\theta(h,y) \delta(h,y))} \\
    &= \theta(h,x) ~{\sigma_{\delta(h,x)x}(\theta(h,y))} \delta(h,x) {^x\delta(h,y)}
\end{align*}
which implies 
\begin{align}
    \theta(h,xy)=\theta(h,x) \sigma_{\delta(h,x)x}(\theta(h,y)) ~~\text{and}~~ \delta(h,xy)= \delta(h,x) ~{^x\delta(h,y)}.
\end{align}

\item Since $hk \star x = {^h(k \star x)} (h \star x)$, we have 
\begin{align*}
    \theta(hk,x)\delta(hk,x) &= {^h(\theta(k,x)\delta(k,x))} (\theta(h,x)\delta(h,x)) \\
    &= h\theta(k,x) ~{\sigma_{\delta(k,x)}(h^{-1}\theta(h,x))} \delta(k,x) \delta(h,x)
\end{align*}
which implies 
\begin{align}
    \theta(hk,x)=  h\theta(k,x) ~{\sigma_{\delta(k,x)}(h^{-1}\theta(h,x))} ~~\text{and}~~ \delta(hk,x)= \delta(k,x) \delta(h,x).
\end{align}

\item Consider the expressions 
\begin{align*}
    &((h \star x) \star {^xy}) = \theta(h,x) f(\delta(h,x), {^xy}) \sigma_{(\delta(h,x) \widetilde{\star} {^xy})}(\theta(h,x)^{-1}\theta( \theta(h,x), {^xy})) (\delta(h,x) \widetilde{\star} {^xy}) \delta(\theta(h,x), {^xy}),\\
   & ((x \star y) \star {\sigma_y(h)})= f(x,y) \sigma_{(\delta({\sigma_y(h)}, x \widetilde{\star} y)^{-1})}(\theta({\sigma_y(h)},x \widetilde{\star} y )^{-1} f(x,y)^{-1}(f(x,y) \star {\sigma_y(h)})) \delta({\sigma_y(h)}, x \widetilde{\star} y)^{-1},\\
   & ((y \star h) \star {^hx}) = {\sigma_{(\delta(h,y)^{-1})}((\theta(h,y)^{-1} \star [h,x]) [h,x]\theta(\theta(h,y)^{-1},x) \sigma_{\delta(\theta(h,y)^{-1},x)}([h,x]^{-1}))} \\
   & \sigma_{^{\delta(h,y)^{-1}}(\delta(\theta(h,y)^{-1},x)) ~ \delta([h,x], \delta(h,y)^{-1})^{-1}}(\theta([h,x],\delta(h,y)^{-1})^{-1} [h,x] f(\delta(h,y)^{-1},x) \sigma_{(\delta(h,y)^{-1} \widetilde{\star} ~x)}([h,x]^{-1})) \\& \sigma_{\delta(h,y)^{-1}}(\delta(\theta(h,y)^{-1},x)) ~ \delta([h,x], \delta(h,y)^{-1})^{-1} (\delta(h,y)^{-1}\widetilde{\star} ~x),
\end{align*}
Since $((h \star x) \star {^xy})((x \star y) \star {^yh})((y \star h) \star {^hx})=1$, we have
\begin{align}
  & \hspace{-1cm} \theta(h,x) f(\delta(h,x), {^xy}) \sigma_{(\delta(h,x) \widetilde{\star} {^xy})}(\theta(h,x)^{-1}\theta( \theta(h,x), {^xy})) \notag\\ & \hspace{-1cm} \sigma_{(\delta(h,x) \widetilde{\star} {^xy}) \delta(\theta(h,x), {^xy})} (f(x,y) \sigma_{\delta({\sigma_y(h)}, x \widetilde{\star} y)^{-1}}(\theta({\sigma_y(h)},x \widetilde{\star} y )^{-1} f(x,y)^{-1}(f(x,y) \star {\sigma_y(h)}))) \notag\\ & \hspace{-1cm} \sigma_{(\delta(h,x) \widetilde{\star} {^xy}) \delta(\theta(h,x), {^xy}) \delta({\sigma_y(h)}, x \widetilde{\star} y)^{-1}}({\sigma_{\delta(h,y)^{-1}}((\theta(h,y)^{-1} \star [h,x]) [h,x]\theta(\theta(h,y)^{-1},x)  \sigma_{\delta(\theta(h,y)^{-1},x)}([h,x]^{-1}))} \notag\\ & \hspace{-1cm}
   \sigma_{^{\delta(h,y)^{-1}}(\delta(\theta(h,y)^{-1},x)) ~ \delta([h,x], \delta(h,y)^{-1})^{-1}}(\theta([h,x],\delta(h,y)^{-1})^{-1} [h,k] f(\delta(h,y)^{-1},x) \sigma_{(\delta(h,y)^{-1} \widetilde{\star} ~x)}([h,x]^{-1}))) =1 \notag \\ 
~~\text{and}~~
  & {(\delta(h,x) \widetilde{\star} ~{^xy}) \delta(\theta(h,x), {^xy}) \delta({\sigma_y(h)}, x \widetilde{\star} y)^{-1}} \sigma_{\delta(h,y)^{-1}}(\delta(\theta(h,y)^{-1},x)) \notag\\& \delta([h,x], \delta(h,y)^{-1})^{-1} (\delta(h,y)^{-1}\widetilde{\star} ~x)=1.  
\end{align}
Similarly, $((h \star k) \star {^kx})((k \star x) \star {^xh})((x \star h) \star {^hk})=1$, thus we have
\begin{align}
    &(h \star k \star [k,x]) [k,x] \theta(h \star k, x) \sigma_{\delta(h \star k, x)}([k,x]^{-1} \theta(k,x) \sigma_{\delta({\sigma_x(h)},\delta(k,x))^{-1}}(\theta({\sigma_x(h)},\delta(k,x))^{-1}  \theta(k,x)^{-1} \notag\\&(\theta(k,x) \star {\sigma_x(h)}))) ~ {\sigma_{\delta(h \star k, x) \delta({\sigma_x(h)},\delta(k,x))^{-1}}({\sigma_{\delta(h,x)^{-1}}(\theta(h,x)^{-1} \star {^hk})}  ~{\sigma_{\delta({^hk}, \delta(h,x)^{-1})^{-1}}(\theta({^hk}, \delta(h,x)^{-1})^{-1}))}}=1 \notag\\
&~~\text{and}~~
    \delta(h \star k, x) \delta({\sigma_x(h)},\delta(k,x))^{-1} \delta({^hk},\delta(h,x)^{-1})^{-1}=1. 
\end{align}


\item Since $^y(h \star x)= {^yh} \star {^yx}$, we have
\begin{align*}
    ^y(\theta(h,x) \delta(h,x))= \theta({\sigma_y(h)},{^yx}) \delta({\sigma_y(h)},{^yx}) \\
    {\sigma_y(\theta(h,x))} {^y\delta(h,x)}= \theta({\sigma_y(h)},{^yx}) \delta({\sigma_y(h)},{^yx})
\end{align*}
which implies 
\begin{align}
    {\sigma_y(\theta(h,x))}=\theta({\sigma_y(h)},{^yx}) ~~\text{and}~~ {^y\delta(h,x)}= \delta({\sigma_y(h)},{^yx}).
\end{align}

Similarly, $^h(k \star x)= {^hk} \star {^hx}$, we have
\begin{align*}
    {^h\theta(k,x)} {^h\delta(k,x)}&= ({^hk} \star [h,x] x)\\
    {^h\theta(k,x)} [h,\delta(k,x)] \delta(k,x)&= ({^hk} \star [h,x]) ^{[h,x]}({^hk \star x}) \\
    &= ({^hk} \star [h,x]) ^{[h,x]}\theta({^hk},x) ^{[h,x]}\delta({^hk},x) \\ 
    &= ({^hk} \star [h,x]) ^{[h,x]}\theta({^hk},x) [[h,x],\delta({^hk},x)] \delta({^hk},x)
\end{align*}
which implies 
\begin{align}
    {^h\theta(k,x)} [h,\delta(k,x)]= ({^hk} \star [h,x]) ^{[h,x]}\theta({^hk},x) [[h,x],\delta({^hk},x)] ~~\text{and}~~ \delta(k,x)=\delta({^hk},x).
\end{align}
\begin{remark}
\begin{itemize}
    \item We can see that \begin{align*}
        (hx) \star' (ky)= (h \sigma_{\delta(k,x)^{-1}}(\theta(k,x)^{-1}h^{-1}(h \star k)kh f(x,y) \notag\\\sigma_{x\widetilde{\star}y}(h^{-1}\theta(h,y)\sigma_{\delta(h,y)}(k^{-1})))) (\delta(k,x)^{-1}(x \widetilde{\star} y)\delta(h,y)). 
    \end{align*}
\item $S$ is a group with an extra binary operation $\widetilde{\star}$ satisfying the following identity:
\begin{enumerate}
\item $ x\widetilde{\star} x= 1$,
    \item $ x\widetilde{\star} yz= (x\widetilde{\star} y)~{^y(x\widetilde{\star} z)},$
            \item $(xy\widetilde{\star} z)={^x(y\widetilde{\star} z)}(x\widetilde{\star} z),$
                         \item $((x\widetilde{\star} y) \widetilde {\star} {^yz})\delta(f(x,y), {^yz})  ((y\widetilde{\star} z) \widetilde {\star} {^zx})\delta(f(y,z), {^zx})((z\widetilde{\star} x) \widetilde {\star} {^xy})\delta(f(z,x), {^xy})=1,$

\item ${^z(x\widetilde{\star} y)}=({^zx} \widetilde{\star} {^zy}),$
\end{enumerate} for all $x, y, z \in S.$
\end{itemize}
\end{remark}



\end{enumerate}

\begin{definition}
    Let $H$ be a multiplicative Lie algebra with multiplicative Lie algebra structure $\star$ and $S$ a group. An extending multiplicative Lie algebra of $H$ through $S$ is a system $\Omega(H,S)=(\sigma,\delta,\theta,f,\widetilde{\star})$ consisting of a group homomorphism $\sigma: S \to Aut(H)$ and four other maps $\delta: H \times S \to S, \theta:H \times S \to H, f: S \times S \to H, \widetilde{\star}: S \times S\to S $. Let $\Omega(H,S)=(\sigma,\delta,\theta,f,\widetilde{\star})$ be an extending multiplicative Lie algebra. We denote $H \rtimes_{\sigma}^{ML}S$ the group $H \rtimes_{\sigma}S$ with the binary operation $\star'$ on $H \rtimes_{\sigma}S$ defined by
    \begin{align}\label{4.12}
        (h,x) \star' (k,y)= (h \sigma_{\delta(k,x)^{-1}}(\theta(k,x)^{-1}h^{-1}(h \star k)kh f(x,y) \notag\\\sigma_{x\widetilde{\star}y}(h^{-1}\theta(h,y)\sigma_{\delta(h,y)}(k^{-1}))), \delta(k,x)^{-1}(x \widetilde{\star} y)\delta(h,y)). 
    \end{align}
 The object $H \rtimes_{\sigma}^{ML}S$ is called the extended multiplicative Lie algebra if it is a multiplicative Lie algebra with the multiplicative Lie algebra structure $\star'$ given by $(\ref{4.12})$.
\end{definition}

Thus, from the above discussion, we have the following theorem:

\begin{theorem}
    Let $H$ be a multiplicative Lie algebra, $S$ a group, and $\Omega(H, S)$ an extending multiplicative Lie algebra of $H$ through $S$. The following statements are equivalent:
    \begin{enumerate}
        \item $H \rtimes_{\sigma}^{ML}S$ is a extended multiplicative Lie algebra.
        \item  The following conditions hold for any $h,k,l \in H$ and $x,y,z\in S$:
        \begin{enumerate}
        \item The pair $(\sigma, \delta)$ is a pseudo action of $H$ on $S$,
            \item $f(x,x)=1$ and $x \widetilde{\star} x=1,$
            \item $ f(x,yz)= f(x,y) {\sigma_{(x\widetilde{\star} y)y}(f(x,z))} ~~\text{and}~~ x\widetilde{\star} yz= (x\widetilde{\star} y)~{^y(x\widetilde{\star} z)},$
            \item $f(xy,z)={\sigma_x(f(y,z))} {\sigma_{^x(y\widetilde{\star} z)}(f(x,z))} ~~\text{and}~~ (xy\widetilde{\star} z)={^x(y\widetilde{\star} z)}(x\widetilde{\star} z),$
             \item $f(x,y) f(x\widetilde{\star} y,{^yz})  \sigma_{((x\widetilde{\star} y) \widetilde{\star} {^yz})}(f(x,y)^{-1}\theta(f(x,y), {^yz})) \sigma_{((x\widetilde{\star} y) \widetilde {\star} {^yz})\delta(f(x,y), {^yz})}(f(y,z) \\ f(y\widetilde{\star} z,{^zx})  \sigma_{((y\widetilde{\star} z) \widetilde{\star} {^zx})}(f(y,z)^{-1}\theta(f(y,z), {^zx})))   \sigma_{((x\widetilde{\star} y) \widetilde {\star} {^yz})\delta(f(x,y), {^yz})  ((y\widetilde{\star} z) \widetilde {\star} {^zx})\delta(f(y,z), {^zx})}(f(z,x) \\ f(z\widetilde{\star} x,{^xy})  \sigma_{((z\widetilde{\star} x) \widetilde{\star} {^xy})}(f(z,x)^{-1}\theta(f(z,x), {^xy}))) =1,$ 

            \item $((x\widetilde{\star} y) \widetilde {\star} {^yz})\delta(f(x,y), {^yz})  ((y\widetilde{\star} z) \widetilde {\star} {^zx})\delta(f(y,z), {^zx})((z\widetilde{\star} x) \widetilde {\star} {^xy})\delta(f(z,x), {^xy})=1,$

\item $\sigma_z(f(x,y))=f({^zx}, {^zy}) ~~\text{and}~~ {^z(x\widetilde{\star} y)}=({^zx} \widetilde{\star} {^zy}),$

\item $\theta(h,xy)=\theta(h,x) \sigma_{\delta(h,x)x}(\theta(h,y)),$
            


\item $ \theta(hk,x)=  h\theta(k,x) ~{\sigma_{\delta(k,x)}(h^{-1}\theta(h,x))}$,
   \item $\theta(h,x)\,
f\bigl(\delta(h,x),{}^{x}y\bigr)\,
\sigma_{\delta(h,x)\widetilde{\star}{}^{x}y}
\Bigl(
\theta(h,x)^{-1}
\theta\bigl(\theta(h,x),{}^{x}y\bigr)
\Bigr)
\\
\quad{}
\sigma_{(\delta(h,x)\widetilde{\star}{}^{x}y)
\delta(\theta(h,x),{}^{x}y)}
\Bigl(
f(x,y)\,
\sigma_{\delta({\sigma_y(h)},\,x\widetilde{\star}y)^{-1}}
\Bigl(
\theta({\sigma_y(h)},x\widetilde{\star}y)^{-1}
f(x,y)^{-1}
\bigl(f(x,y)\star{\sigma_y(h)}\bigr)
\Bigr)
\Bigr)
\\
\quad{}
\sigma_{(\delta(h,x)\widetilde{\star}{}^{x}y)
\delta(\theta(h,x),{}^{x}y)
\delta({\sigma_y(h)},x\widetilde{\star}y)^{-1}}
\Biggl(
\sigma_{\delta(h,y)^{-1}}
\Bigl(
\bigl(\theta(h,y)^{-1}\star[h,x]\bigr)
[h,x]\,
\theta\bigl(\theta(h,y)^{-1},x\bigr)
\\
\hspace{5cm}
\sigma_{\delta(\theta(h,y)^{-1},x)}
\bigl([h,x]^{-1}\bigr)
\Bigr)
\sigma_{
{}^{\delta(h,y)^{-1}}
(\delta(\theta(h,y)^{-1},x))
\,
\delta([h,x],\delta(h,y)^{-1})^{-1}
} \Bigl(
\theta([h,x],\delta(h,y)^{-1})^{-1}
\\
\hspace{2.5cm}
[h,k]
f(\delta(h,y)^{-1},x)\,
\sigma_{\delta(h,y)^{-1}\widetilde{\star}x}
\bigl([h,x]^{-1}\bigr)
\Bigr)
\Biggr)
=1,$
\item ${(\delta(h,x) \widetilde{\star} ~{^xy}) \delta(\theta(h,x), {^xy}) \delta({\sigma_y(h)}, x \widetilde{\star} y)^{-1}} \sigma_{\delta(h,y)^{-1}}(\delta(\theta(h,y)^{-1},x)) \\ \delta([h,x], \delta(h,y)^{-1})^{-1} (\delta(h,y)^{-1}\widetilde{\star} ~x)=1$,
\item
$
(h \star k \star [k,x]) [k,x]\,
\theta(h \star k,x) \sigma_{\delta(h \star k,x)}
\Big(
[k,x]^{-1}
\theta(k,x) \sigma_{\delta({\sigma_x (h)},\delta(k,x))^{-1}} \big(
\theta({\sigma_x (h)},\delta(k,x))^{-1}
\\
\qquad
\theta(k,x)^{-1}
\big(
\theta(k,x)\star {\sigma_x (h)}
\big)
\big)
\Big) 
{}\sigma_{\delta(h \star k,x)\,
\delta({\sigma_x (h)},\delta(k,x))^{-1}}
\Big(
{} \sigma_{\delta(h,x)^{-1}}
\big(
\theta(h,x)^{-1}\star {^h k}
\big)
\\
\qquad
{}\sigma_{\delta({^h k},\delta(h,x)^{-1})^{-1}}
\theta({^h k},\delta(h,x)^{-1})^{-1}
\Big)
=1
$,
\item ${^h\theta(k,x)} [h,\delta(k,x)]= ({^hk} \star [h,x]) ^{[h,x]}\theta({^hk},x) [[h,x],\delta({^hk},x)]$ and $\delta(k,x)=\delta({^hk},x),$ 
\item ${\sigma_y(\theta(h,x))}=\theta({\sigma_y(h)},{^yx}). $ 

\end{enumerate}
where for $h\in H$ and $x\in S$, $[h,x]$ denotes the element $h\sigma_x(h^{-1})$
\end{enumerate}
\end{theorem}

The above theorem is very technical and hard to apply in general. So, we restrict ourselves to the cases when $\sigma$ is trivial and $H$, $S$ are abelian. We have the following remarks.

\begin{remark}\label{4.2} Suppose $\sigma=I_H$.  Then  $H \rtimes_{\sigma}^{ML}S$ is a extended multiplicative Lie algebra 
    if  the following conditions holds  for all $h,k,l\in H$ and $x,y,z\in S$:
        \begin{enumerate}
        \item The pair $(I_H, \delta)$ is a pseudo action of $H$ on $S$,
            \item $f(x,x)=1$ and $x \widetilde{\star} x=1,$
            \item $ f(x,yz)= f(x,y) f(x,z) ~~\text{and}~~ x\widetilde{\star} yz= (x\widetilde{\star} y)~{^y(x\widetilde{\star} z)},$
            \item $f(xy,z)=f(y,z) f(x,z) ~~\text{and}~~ (xy\widetilde{\star} z)={^x(y\widetilde{\star} z)}(x\widetilde{\star} z),$
             \item $f(x,y) f(x\widetilde{\star} y,{^yz})  f(x,y)^{-1}\theta(f(x,y), {^yz}) f(y,z) \\ f(y\widetilde{\star} z,{^zx})  f(y,z)^{-1}\theta(f(y,z), {^zx})   f(z,x) \\ f(z\widetilde{\star} x,{^xy})  f(z,x)^{-1}\theta(f(z,x), {^xy}) =1,$ 

            \item $((x\widetilde{\star} y) \widetilde {\star} {^yz})\delta(f(x,y), {^yz})  ((y\widetilde{\star} z) \widetilde {\star} {^zx})\delta(f(y,z), {^zx})((z\widetilde{\star} x) \widetilde {\star} {^xy})\delta(f(z,x), {^xy})=1,$

\item $f(x,y)=f({^zx}, {^zy}) ~~\text{and}~~ {^z(x\widetilde{\star} y)}=({^zx} \widetilde{\star} {^zy}),$

\item $\theta(h,xy)=\theta(h,x) \theta(h,y),$
\item $ \theta(hk,x)=  {^h\theta(k,x)} \theta(h,x)$,
   \item $\theta(h,x)\,
f\bigl(\delta(h,x),{}^{x}y\bigr)\,
\theta(h,x)^{-1}
\theta\bigl(\theta(h,x),{}^{x}y\bigr)
f(x,y)\,
\theta(h,x\widetilde{\star}y)^{-1}
f(x,y)^{-1}
\bigl(f(x,y)\star h \bigr)
\\
\quad{}
\bigl(\theta(h,y)^{-1}\star[h,x]\bigr)
[h,x]\,
\theta\bigl(\theta(h,y)^{-1},x\bigr)
[h,x]^{-1}
\theta([h,x],\delta(h,y)^{-1})^{-1}
\\
\hspace{2.5cm}
[h,k]
f(\delta(h,y)^{-1},x)\,
[h,x]^{-1}
=1,$
\item ${(\delta(h,x) \widetilde{\star} ~{^xy}) \delta(\theta(h,x), {^xy}) \delta(h, x \widetilde{\star} y)^{-1}} \delta(\theta(h,y)^{-1},x) \\ \delta([h,x], \delta(h,y)^{-1})^{-1} (\delta(h,y)^{-1}\widetilde{\star} ~x)=1$,
\item
$
(h \star k \star [k,x]) [k,x]\,
\theta(h \star k,x) 
[k,x]^{-1}
\theta(k,x)
\theta(h,\delta(k,x))^{-1}
\\
\theta(k,x)^{-1}
\big(
\theta(k,x)\star h
\big)
\big(
\theta(h,x)^{-1}\star {^h k}
\big)
\theta({^h k},\delta(h,x)^{-1})^{-1}
=1,
$
\item ${^h\theta(k,x)} =  \theta({^hk},x)$ and $\delta(k,x)=\delta({^hk},x),$ 
\item $\theta(h,x)=\theta(h,{^yx}).$ 
\end{enumerate}

\end{remark}


\begin{remark} Suppose   $H \rtimes_{\sigma}^{ML}S$ is a extended multiplicative Lie algebra.
\begin{enumerate}
    \item Since $(h, 1) \star' (k, 1) = (h \star k, 1)$. Hence, we can say that $H$ is a subalgebra of $H \rtimes_{\sigma}^{ML}S$.     

\item If $\delta(h, x) =1$ for all $h\in H$ and $x \in S$. Then $$(h, 1) \star' (k, y) = (h \theta(k,x)^{-1}h^{-1}(h \star k)kh f(x,y) \notag h^{-1}\theta(h,y)k^{-1}, 1).$$ This shows that $H$ is an ideal and $\widetilde{\star}$ is multiplicative Lie algebra  structure on $S$ with $G/H \cong S$ as a multiplicative Lie algebra.
\item  If $H$ and $S$ are Lie algebras, then Remark $\ref{4.2}$ is the same as Theorem $2.2$ of \cite{AG}.
\end{enumerate}
\end{remark}

\noindent{\bf Acknowledgment:} We would like to express our sincere thanks to Prof. Guram Donadze, Institute of Cybernetics of Georgian Technical University, for his valuable suggestions. We are also thankful to the National Board for Higher Mathematics (NBHM) for providing the project ``Linear Representation of Multiplicative Lie Algebra" (02011/19/2023/NBHM (R.P)/R~\& D-II/5954). Also, the first-named author thanks IIIT Allahabad and the Ministry of Education, Government of India, for providing the institute fellowship.

\noindent{\bf  Data Availability declaration:} Not applicable.

\noindent{\bf Competing Interests:} The authors have no competing interests to declare that are relevant to the content of this article.

\noindent{\bf Funding Declaration:} The first-named author is supported by the National Board for Higher Mathematics (NBHM), DAE, Govt. of India, project ``Linear Representation of Multiplicative Lie Algebra" (02011/19/2023/NBHM (R.P)/R~\& D-II/5954).


\begin{thebibliography}{9}
     \bibitem{AG1} A. L. Agore and G. Militaru; Extending structures I: the level of groups, Algebr. Represent. Theory 17 (2014), 831-848.
     \bibitem{AG} A. L. Agore and G. Militaru; Extending structures for Lie algebras, Monatsh. Math. 174 (2014), 169-193.
    \bibitem{AGNM} A. Bak, G. Donadze, N. Inassaridze, M. Ladra; Homology of multiplicative Lie rings, J. Pure Appl. Algebra 208 (2007) 761-777.

    
    \bibitem{GNM} G. Donadze, N. Inassaridze, M. Ladra; Non-abelian tensor and exterior products of multiplicative Lie rings, Forum Math. 29 (2017) 563-574.
    \bibitem{GNMA} G. Donadze, N. Inassaridze, M. Ladra, A. M. Vieites; Exact sequences in homology of multiplicative Lie rings and a new version of Stallings Theorem, J. Pure Appl. Algebra 222 (2018) 1786-1802.

\bibitem{GM}
G. Donadze and M. Ladra;
More on five commutator identities, {J. Homotopy Relat. Struct.},  2 (1) ( 2007), 45-55.
    
    \bibitem{GJ}
	G. J. Ellis; On five well-known commutator identities, J. Aust. Math. Soc. (Series A), 54 (1993), 1-19.

    \bibitem{EH}
J. E. Humphreys; Introduction to Lie algebras and representation theory, Springer-Verlag, 1980.

\bibitem{KJ} A. Kumar, R. Joshi, M. S. Pandey, S. K. Upadhyay; The Bogomolov multiplier of a multiplicative Lie algebra, Ricerche Mat., 75 (2) (2026), 1083-1097.

    \bibitem{KU} A. Kumar, S. Kushwaha, and S. K. Upadhyay; The multiplicative Lie algebra on general linear groups, Georgian Math. J., 32 (3) (2025), 457-463. 

    \bibitem{KP} A. Kumar, D. Pal,  S. Kushwaha, and S. K. Upadhyay; {The Schur multiplier of some finite multiplicative Lie algebras}, J. Lie theory (2024), 34 (2), 423-436.

\bibitem{RS} R. Lal and S. K. Upadhyay; Multiplicative Lie algebra and Schur multiplier, J. Pure Appl. Algebra 223 (9) (2019), 3695-3721.

\bibitem{MKU} N. K. Maurya, A. Kumar, and S. K. Upadhyay; {Excision and idealization of a multiplicative Lie algebra}, Ricerche Mat. (2026), https://doi.org/10.1007/s11587-026-01087-8.

\bibitem{PK} D. Pal, A. Kumar, and S. K. Upadhyay;  {Multiplicative Lie algebra structure on a nilpotent group of class $2$}, Georgian Math. J. (2026), 33 (2), 301-306.

\bibitem{DSS}  D. Pal, A. Kumar, S. K. Upadhyay, and S. Kushwaha;
Multiplicative Lie algebra structures on the semi-direct product of groups, J. Lie Theory 35 (2025), no. 3, 573-582.

\bibitem{MRS}
M. S. Pandey, R. Lal, and S. K. Upadhyay;
Lie commutator, solvability and nilpotency in multiplicative Lie algebras, J. Algebra Appl., 20 (8) (2021), 2150138 (11 pp).
    
	\bibitem{MS}	M. S. Pandey and S. K. Upadhyay; Theory of extension of multiplicative Lie algebras, Journal of Lie Theory 31 (2021) 637-658.
    \bibitem{MS1}
M. S. Pandey and S. K. Upadhyay; Classification of multiplicative Lie algebra structures on a finite group, Colloq. Math., 168 (2022), 25-34.

\bibitem{FP}
F. Point and P. Wantiez;
Nilpotency criteria for multiplicative Lie algebra, J. Pure. Appl. Algebra, 111 (1996), 229-243.


\bibitem{BS}
B. Steinberg; Representation theory of finite groups, Springer Newyork, 2012.

\bibitem{GW}
G. L. Walls;
Multiplicative Lie algebras, {Turkish J. Math.}, 43 (2019), 2888–2897.



\end{thebibliography}
\end{document}